\documentclass[12pt]{amsart}
\usepackage{amsmath}
\usepackage{amssymb}
\usepackage{amsthm}
\usepackage[final]{hyperref}
\usepackage{graphicx}
\newtheorem{thm}{Theorem}
\newtheorem{lem}[thm]{Lemma}
\newtheorem{cor}[thm]{Corollary}
\newtheorem{prop}[thm]{Proposition}
\theoremstyle{definition}

\theoremstyle{definition}

\newcommand{\R}{\mathbb{R}^{d}}

\title[Gradient Bounds and Coarsening]{Gradient Bounds and Coarsening Estimates for Some Molecular Beam Epitaxy Models}
\begin{document}
\author{Daniel Oliveira da Silva}
\address{Department of Mathematics \\
Nazarbayev University \\
Qabanbai Batyr Avenue 53 \\
010000 Nur-Sultan \\
Republic of Kazakhstan}
\email{daniel.dasilva@nu.edu.kz}
\author{Magzhan Biyar}
\address{Department of Mathematics \\
Nazarbayev University \\
Qabanbai Batyr Avenue 53 \\
010000 Nur-Sultan \\
Republic of Kazakhstan}
\email{magzhan.biyarov@nu.edu.kz}

\begin{abstract}
  We prove global well-posedness for some molecular beam epitaxy models.  This requires proving some bounds on the gradient of solutions.  As a consequence of these estimates, we also obtain bounds on the coarseness of solutions.
\end{abstract}


\maketitle

\section{Introduction}

Molecular beam epitaxy (MBE) is a method of depositing very thin layers of crystals onto surfaces such that the orientation of the crystals relative to the substrate is well-defined.  Such manufacturing methods are used in the fabrication of various types of objects, such as optical components and semiconductor devices.  In many models of epitaxial dynamics, the profile of the thin film is described by a partial differential equation of the general form
\begin{equation}\label{MBE}
u_{t} + \Delta^{2} u + \nabla \cdot \mathbf{J}(\nabla u) = 0.
\end{equation}
Here, the function $u : \Omega \times \mathbb{R} \longrightarrow \mathbb{R}$ is called the \emph{height function}, and represents the height of the thin film layer in a comoving reference frame.  The set $\Omega \subset \mathbb{R}^{2}$ represents the surface being coated, and the vector function $\mathbf{J}$  is called the \emph{surface diffusion current}.  This equation is similar to the Cahn-Hilliard equation
\[
u_{t} + \Delta^{2} u + \Delta F(u) = 0.
\]
However, it is known that there are differences in the behavior of solutions to these two equations; see, for example, \cite{SP1994, S1997, S1998} for discussions.

In the modeling of epitaxial processes with equations such as \eqref{MBE}, there are two questions of fundamental importance:
\begin{itemize}
\item Well-posedness: Do unique solutions to \eqref{MBE} exist for a given class of initial data?
\item Coarsening: How much do solutions to equation \eqref{MBE} deviate from the baseline?
\end{itemize}
Typically, the coarseness $C(t)$ is taken to be the root-mean-square deviation from the baseline and is measured by the $L^{2}$ norm of the solution; see the discussion in section 3 of \cite{B2020} for more details about coarseness in deterministic MBE models, or the introduction in \cite{BMW2001} for stochastic models.

In the present work, we will consider these questions for models of the form of equation \eqref{MBE} for surface currents $\mathbf{J}$ satisfying
\[
\mathbf{J}(\nabla u) = O(|\nabla u|^{q})
\]
for $q \geq 1$.  This class contains many well-known models, such as the model of Rost and Krug \cite{RK1997} (also some special cases of the model of Villain \cite{V1991}),
\[
\mathbf{J}(\nabla u) = \left( 1 - |\nabla u|^{2} \right)\nabla u;
\]
and the models of Johnson, Orme, Hunt, Graff, Sudijono, Sander, and Orr \cite{JOHGSSO1994} and Siegert and Plischke \cite{SP1994}, where $\mathbf{J}$ is given by
\[
\mathbf{J}(\nabla u) = (u_{x}f(u_{x}), u_{y}f(u_{y}))
\]
for continuous rational functions $f$.  For such models, we will prove the following global well-posedness result:
\begin{thm}\label{mainthm1}
Let $1 \geq p$, let $\mathbf{J}$ be a differentiable function such that $\mathbf{J}(\nabla u) = O( | \nabla u |^{q})$.  Then there exists a number $\epsilon_{0} > 0$ such that the Cauchy problem
\begin{equation}\begin{aligned}\label{Cauchy}
& u_{t} + \Delta^2 u + \nabla \cdot \mathbf{J}(\nabla u) = 0 \\
& u(0) = u_0.
\end{aligned}\end{equation}
is globally well-posed in $W^{1,p} \cap W^{1,\infty}(\R)$ for small initial data whenever
\[
p = \frac{d(q-1)}{2} \quad \text{and} \quad q > \frac{5}{3}.
\]
\end{thm}
As we will see, this theorem is stated in a slightly stronger form than necessary; the smallness assumption applies only to the gradient of solutions, and not to the solution itself.

To prove this result, we first obtain a local well-posedness result in section \ref{localwell}.  We will then show that it is possible to control the norm of the solutions using the following a priori estimate on the gradient of solutions:
\begin{thm}\label{mainthm2}
Let $u$ be the solution to the Cauchy problem
\begin{align*}
& u_{t} + \Delta^{2} u + \nabla \mathbf{J}(\nabla u) = 0 \\
& u(x,0) = u_{0}
\end{align*}
with $\mathbf{J}(\nabla u) = O( |\nabla u|^{q} )$.  Then there exists a constant $\epsilon_{0} > 0$ such that if
\[
\| \nabla u_{0} \|_{L^{p}} \leq \epsilon_{0},
\]
then
\[
\| \nabla u(t) \|_{L^{\infty}_{x}} \leq \frac{C_{1}}{t^{\frac{d}{4p}}}.
\]
and
\[
\| \nabla u(t) \|_{L^{p}_{x}} \leq C_{2}.
\]
for some positive constants $C_{1}, C_{2}$, and for
\[
p = \frac{d(q-1)}{2}.
\]
\end{thm}
\noindent This will be proved in section \ref{gradboundsec}.  We remark that this result should be compared to the result of Li, Qiao, and Tang in \cite{LQT2017}, who obtained a similar result for the problem on the torus.  Once we prove Theorem \ref{mainthm2}, we may obtain the necessary control on the norms of the solution.  This will be done in section \ref{globalwell}.  

Our final main result gives us a maximum rate of growth on the coarseness of solutions:
\begin{thm}\label{mainthm3}
Let $u \in C([0,\infty); W^{1,p}(\R) \cap W^{1,\infty}(\R))$ be the global solution to the Cauchy problem \eqref{Cauchy} from Theorem \ref{mainthm1} for
\[
\max \left\{ 1 + \frac{2}{d}, \frac{6+d}{2+d} \right\} < q < 1 + \frac{4}{d}.
\]
Then the coarseness $$C(t) = \| u(t) \|_{L^{2}}$$ of $u$ satisfies
\[
C(t) \leq C(1+t)^{\frac{1}{4} - \frac{d}{4p} + \frac{d}{8}}.
\]
for all $t \in [0, \infty)$.
\end{thm}
This will be shown in section \ref{rough}.  We remark that in the special case $d = 2$ (which is the one of greatest interest in applications), the above result reduces to
\[
\| u(t) \|_{L^{2}} \leq C(1+t)^{\frac{1}{2} - \frac{1}{p}}
\]
when $2 < q < 3$.

\section{Preliminaries}\label{prel}

In this section, we state some preliminary results which will be used in the proofs of Theorems \ref{mainthm1}, \ref{mainthm2} and \ref{mainthm3}.  First, we recall some basic facts about Cahn-Hilliard-type equations.  Consider the problem
\begin{equation}\label{cahnhilliard}\begin{aligned}
& u_{t} + \Delta^2 u = G(x,t), \\
& u(x,0) = f(x),
\end{aligned}
\end{equation}
with $(x,t) \in \mathbb{R}^{d} \times \mathbb{R}_{+}$.  By elementary Fourier analysis, the solution is given by
\begin{equation}\label{solution}
u(x,t) = \int_{\R} k_{t}(x-y) f(y)\ dy + \int_{0}^{t} \int_{\R} k_{t-s}(x-y) G(y,s)\ dyds,
\end{equation}
where the kernel $k_{t}(x)$ is the function whose spatial Fourier transform is
\[
\hat{k}_{t}(\xi) = e^{- t |\xi|^4}.
\]
We observe that the expression in equation \eqref{solution} can be expressed in the more succinct form
\[
u(x,t) = [k_{t} * f](x) + \int_{0}^{t} [k_{t-s} * G(\cdot, s)](x)\ ds,
\]
where $*$ denotes the convolution
\[
f*g = [f*g](x) = \int_{\R} f(x-y) g(y)\ dy.
\]
For later convenience, we introduce the following convention: if $\mathbf{f} = (f_{1}, \cdots, f_{n})$ and $\mathbf{g} = (g_{1}, \cdots, g_{n})$ are vector functions, we write
\[
\mathbf{f}*\mathbf{g} = \int_{\R} \mathbf{f}(x-y) \cdot \mathbf{g}(y)\ dy,
\]
where $\mathbf{f} \cdot \mathbf{g}$ denotes the standard dot product of the vector functions $\mathbf{f}$ and $\mathbf{g}$.

Applying this to the nonlinear problem \eqref{MBE}, we recast our problem in the integral form
\[
u(x,t) = [k_{t} * u_{0}](x) + \int_{0}^{t} [k_{t-s} * \nabla \cdot \mathbf{J}(\nabla u)](x)\ ds.
\]
Note that by properties of convolutions, this may be rewritten in the slightly different form
\begin{equation}\label{integraleq}
u(x,t) = [k_{t} * u_{0}](x) + \int_{0}^{t} [\nabla k_{t-s} * \mathbf{J}(\nabla u)](x)\ ds.
\end{equation}
Following the nomenclature of Tao \cite{T2006}, we call solutions to the integral equation \eqref{integraleq} which belong to the space $C( I ; Z )$ for some function space $Z$ and time interval $I$ a \emph{strong} $Z$ \emph{solution}.

To find solutions to equation \eqref{integraleq}, we will need to make use of properties of the function $k_{t}$, which we collect in the following lemma:
\begin{lem}\label{kernel}
Assume that $\hat{k}_{t}(\xi) = e^{- t |\xi|^{4}}$,  $t > 0$, and $\xi, x \in \mathbb{R}^{d}$.  Then we have the estimates
\begin{align}
    & \quad \ \| k_{t} \|_{L^{p}_{x}} \leq C_{p} t^{-\frac{d}{4}\left( 1 - \frac{1}{p} \right)}, \label{kest1} \\
    & \| D^{n} k_{t} \|_{L^{p}_{x}} \leq C_{p,n} t^{-\frac{d}{4}\left( 1 - \frac{1}{p} \right) - \frac{n}{4}}. \label{kest2}
\end{align}
Here, $C_{p}$ and $C_{p,n}$ are positive constants with $C_{1} = 1$, and $D^{n} k_{t}$ denotes the set of all derivatives of $k_{t}$ of order $n$.
\end{lem}
\noindent For a proof, see \cite{LWZ2007}, Lemma 2.1.  

Next, we state some lemmas which will be useful in establishing the global theory.  The first is the following result, which will be useful in the global theory:
\begin{prop}[Integrable Bihari Inequality]\label{Bihari}
Let $g$ be continuous on $[a,b]$, let $h(t)$ be integrable on $[a,b]$, and assume that $$k \geq 0 \quad \text{and} \quad M \geq 0.$$  Suppose that the function $\omega(u)$ is a non-negative, non-decreasing continuous function for $u \geq 0$.  If $g$ satisfies the inequality
\[
g(t) \leq k + M \int_{a}^{t} h(s) \omega\left( g(s) \right)\ ds,
\]
then
\[
g(t) \leq \Omega^{-1}\left( \Omega(k) + M \int_{a}^{t} h(s)\ ds \right),
\]
where
\[
\Omega(u) = \int_{u_{0}}^{u} \frac{dy}{\omega(y)},
\]
where $u_{0} > 0$, $u \geq 0$.  Note that the expression $\Omega(k)$ must be well-defined, and $t$ must belong to some subinterval $[a, b']$ of $[a,b]$ so that the expression $\Omega(k) + M \int_{a}^{t} h(s)\ ds$ is in the domain of $\Omega^{-1}$.
\end{prop}
The proof will require basic result from integration theory (see, for example, \cite{R1988}, page 107):
\begin{lem}\label{diff}
Let $f$ be integrable on $[a,b]$, and suppose that
\[
F(t) = F(a) + \int_{a}^{t} f(s)\ ds.
\]
Then $F'(t) = f(t)$ for almost every $t \in [a,b]$.
\end{lem}
\begin{proof}[Proof of Proposition \ref{Bihari}]
The proof is a modification of Bihari's original argument in \cite{B1956}.  Define
\[
V(t) = k + M \int_{a}^{t} h(s) \omega(g(s))\ ds,
\]
from which we observe that $V(a) = k$.  Since $h$ is integrable and $g$ and $\omega$ are continuous, then the product $h(s) \omega(g(s))$ is integrable.  Applying Lemma \ref{diff}, we have that
\begin{equation}\label{differential}
V'(t) = M h(t) \omega(g(t))
\end{equation}
for almost every $t \in [a,b]$.

Next, observe that $g(t) \leq V(t)$ by construction.  Since $\omega$ is non-decreasing, then we must have that $\omega(g(t)) \leq \omega(V(t))$, which implies that
\[
\frac{M h(t) \omega(g(t))}{\omega(V(t))} \leq M h(t).
\]
Applying equation \eqref{differential}, this becomes
\[
\frac{V'(t)}{\omega(V(t))} \leq M h(t)
\]
for almost every $t$.  Integrating this expression then yields
\[
\int_{a}^{t} \frac{V'(s)}{\omega(V(s))}\ ds \leq M\int_{a}^{t} h(s)\ ds.
\]
We may rewrite the integral on the left as
\begin{align*}
\int_{a}^{t} \frac{V'(s)}{\omega(V(s))}\ ds & = \int_{V(a)}^{V(t)} \frac{dV}{\omega(V)} \\ & = \Omega(V(t)) - \Omega(V(a)) \\ & = \Omega(V(t)) - \Omega(k).
\end{align*}
Hence, we obtain
\[
\Omega(V(t)) \leq \Omega(k) + M\int_{a}^{t} h(s)\ ds.
\]
To conclude the proof, we observe that if $\omega$ is non-decreasing, then so are $\Omega$ and $\Omega^{-1}$.  The desired result follows.
\end{proof}
In addition to the Bihari-type inequality above, we will also need the following result, which is Lemma 3.7 of \cite{S1968}:
\begin{lem}\label{strauss}
Let $M(t)$ be a non-negative continuous function of $t$ satisfying the inequality
\[
M(t) \leq c_{1} + c_{2} [M(t)]^{\gamma}
\]
in some interval containing 0, where $c_{1}$ and $c_{2}$ are positive constants and $\gamma > 1$.  If $M(0) \leq c_{1}$ and
\[
c_{1} c_{2}^{(\gamma - 1)^{-1}} < (1-\gamma^{-1})\gamma^{-(\gamma - 1)^{-1}},
\]
then in the same interval,
\[
M(t) < \frac{c_{1}}{1 - \gamma^{-1}}.
\]
\end{lem}

To finish our preliminary discussion, we collect some basic results which will be useful below.
\begin{lem}\label{beta}
Let $a < 1$ and $b < 1$.  Then there exists a constant $C_{a,b}$ such that
\[
\int_{0}^{t} (t-s)^{-a}s^{-b}\ ds = C_{a,b} t^{1 - a - b}.
\]
\end{lem}
\begin{proof}
By the simple substitution $u = \frac{s}{t}$, we have
\[
\int_{0}^{t} (t-s)^{-a}s^{-b}\ ds = t^{1-a-b}\int_{0}^{1} (1-u)^{-a} u^{-b}\ du.
\]
The integral on the right converges when $a < 1$ and $b < 1$.
\end{proof}
\begin{lem}[Young's Inequality]\label{young}
Let $f \in L^{p}(\mathbb{R}^{d})$ and $g \in L^{q}(\mathbb{R}^{d})$, for $1 \leq p, q \leq \infty$ with
\[
\frac{1}{p} + \frac{1}{q} \geq 1.
\]
Then $f * g \in L^{r}(\mathbb{R}^{d})$, where
\[
\frac{1}{r} = \frac{1}{p} + \frac{1}{q} - 1.
\]
Moreover, we have the estimate
\begin{equation}\label{youngineq}
\| f * g \|_{L^{r}} \leq \| f \|_{L^{p}} \| g \|_{L^{q}}.
\end{equation}
\end{lem}

\begin{lem}[Gagliardo-Nirenberg Inequality]\label{gagl}
Let $1 < p < q \leq \infty$, and let $s \geq 0$ be such that
\[
\frac{1}{q} = \frac{1}{p} - \frac{\theta s}{d}
\]
for some $0 < \theta < 1$.  Then for any $u \in W^{1,p}(\R)$, we have
\[
\| u \|_{L^{q}} \leq C \| u \|_{L^{p}}^{1-\theta} \| u \|_{\dot{W}^{s,p}}^{\theta}
\]
for some $C > 0$.
\end{lem}

\section{Local Well-Posedness}\label{localwell}

We now turn to the proof of Theorem \ref{mainthm1}.  As is standard practice, we begin by first obtaining a local result.  We reformulate our local result in the following more explicit (and more general) form:
\begin{prop}\label{local}
Let $1 \leq p < \infty$, and assume that the surface current satisfies $\mathbf{J}(\nabla u) \leq C |\nabla u|^{q}$ for some $q \geq 1$.  Then for any $$u_{0}^{*} \in W^{1,p}(\R) \cap W^{1,\infty}(\R),$$ there exists a time $\delta > 0$ and an open ball $B \subset W^{1,p}(\R) \cap W^{1,\infty}(\R) $ containing $u_{0}^{*}$, such that for each $u_{0} \in B$ there exists a unique strong $W^{1,p}(\R) \cap W^{1,\infty}(\R)$ solution $u$ to the integral equation \eqref{integraleq} in the interval $I = [0, \delta)$.  Furthermore, the solution map $u_{0} \mapsto u$ is continuous from $B$ to $C(I;W^{1,p} \cap W^{1,\infty}(\R))$.
\end{prop}

We prove this by a standard iteration argument.  Fix $u_{0}^{*} \in W^{1,p}(\R)$, and let $B = B(u_{0}^{*}, R)$ be the open ball of radius $R$ centered at $u_{0}^{*}$.  Define a sequence $\{ u^{n} \}_{n = 0}^{\infty}$ by 
\begin{align*}
& u^{0}_{t} + \Delta^{2} u^{0} = 0 & \quad & u^{n}_{t} + \Delta^{2} u^{n} + \nabla \cdot \mathbf{J}(\nabla u^{n-1}) = 0 \\
& u^{0}(0) = u_{0} & & u^{n}(0) = u_{0}.
\end{align*}
We will show that this sequence converges in $C\left( [0,\delta); W^{1,p} \cap W^{1,\infty}\right)$ for sufficiently small $\delta$.  As a first step, we show it is bounded.
\begin{prop}\label{bound}
If $\delta > 0$ is sufficiently small, then
\[
\| u^{n} \|_{L_{t}^{\infty}(W^{1,p} \cap W^{1,\infty})} \leq 2 \| u_{0} \|_{W^{1,p} \cap W^{1,\infty}}
\]
for each $n \geq 0$.
\end{prop}
\begin{proof}
By induction.  For $n = 0$, we apply equation \eqref{solution} with $G = 0$ to obtain
\[
\| u^{0}(t) \|_{W^{1,p}} = \| k_{t} * u_{0} \|_{W^{1,p}} \leq \| k_{t} \|_{L^{1}} \| u_{0} \|_{W^{1,p}}.
\]
If we now apply Lemma \ref{kernel}, this simplifies to
\[
\| u^{0}(t) \|_{W^{1,p}} \leq \| u_{0} \|_{W^{1,p}},
\]
which implies the desired result for $n = 0$.

Now assume the result holds for $n = m$.  Then applying equation \eqref{solution} once again, we see that
\begin{equation}\label{ay}\begin{aligned}
\| u^{m+1} \|_{W^{1,p}} & \leq \| k_{t} * u_{0} \|_{W^{1,p}} + \int_{0}^{t} \| k_{t-s} * \nabla \cdot \mathbf{J}(\nabla u^{m}) \|_{W^{1,p}}\ ds \\
& \leq \| u_{0} \|_{W^{1,p}} + \int_{0}^{t} \| \nabla k_{t-s} \|_{W^{1,1}} \| \mathbf{J}(\nabla u^{m}) \|_{L^{p}} \ ds 
\end{aligned}\end{equation}
Using the assumption on $J$, we have that
\begin{equation}\label{bee}\begin{aligned}
\| \mathbf{J}(\nabla u^{m}(s)) \|_{L^{p}} & \leq \| \nabla u^{m}(s) \|_{L^{\infty}}^{q-1} \| \nabla u^{m}(s) \|_{L^{p}} \\
& \leq \| u(s) \|_{W^{1,\infty}}^{q-1} \| u(s) \|_{W^{1,p}} \\
& \leq \| u(s) \|_{W^{1,p} \cap W^{1,\infty}}^{q} \\
& \leq (2 \| u_{0} \|_{W^{1,p} \cap W^{1,\infty}})^{q}
\end{aligned}\end{equation}
Using the properties of $k_{t-s}$ from Lemma \ref{kernel}, we see that
\begin{equation}\label{cee}
\| \nabla k_{t-s} \|_{W^{1,1}} \lesssim (t-s)^{-\frac{1}{2}} + (t-s)^{-\frac{1}{4}}.
\end{equation}
Combining equations \eqref{ay}, \eqref{bee}, and \eqref{cee}, we see that
\[
\| u^{m+1}(t) \|_{W^{1,p}} \leq \| u_{0} \|_{W^{1,p}} + C(t^{\frac{1}{2}} + t^{\frac{3}{4}})(2 \| u_{0} \|_{W^{1,p} \cap W^{1,\infty}})^{q}.
\]
Taking the supremum in $I = [0,\delta)$ gives us
\[
\| u^{m+1} \|_{L^{\infty}_{t}W^{1,p}_{x}} \leq \| u_{0} \|_{W^{1,p}} + C(\delta^{\frac{1}{2}} + \delta^{\frac{3}{4}})(2 \| u_{0} \|_{W^{1,p} \cap W^{1,\infty}})^{q}.
\]
The desired result follows for $\delta$ sufficiently small.  To complete the proof, we simply observe that the computations above also hold in the special case $p = \infty$ with only a minor modification.
\end{proof}
\begin{prop}\label{convergence}
The sequence $\{ u^{n} \}_{n = 0}^{\infty}$ satisfies
\[
\| u^{n} - u^{n-1} \|_{L^{\infty}_{t}( W^{1,p} \cap W^{1,\infty})} \leq C \delta^{\frac{n}{2}}
\]
for some $C > 0$.
\end{prop}
\begin{proof}
By induction on $n$.  For $n = 1$, we apply a similar computation as in the proof of Proposition \ref{bound} to say that
\[
\| u^{1} - u^{0} \|_{L^{\infty}_{t} W^{1,p}_{x}} \leq C(\delta^{\frac{1}{2}} + \delta^{\frac{3}{4}})(2 \| u_{0} \|_{W^{1,p} \cap W^{1,\infty}})^{q}.
\]
If $\delta$ is sufficiently small, this reduces to
\[
\| u^{1} - u^{0} \|_{L^{\infty}_{t}W^{1,p}_{x}} \leq 2^{q+1}C \delta^{\frac{1}{2}} \| u_{0} \|_{W^{1,p} \cap W^{1,\infty}}^{q}.
\]
This proves the result for $n = 1$.

Now assume the result holds for $n = m$.  Applying equations \eqref{solution} and \eqref{kest1}, we obtain that
\[
\| u^{m+1} - u^{m} \|_{W^{1,p}} \leq \int_{0}^{t} \| \nabla k_{t-s} \|_{W^{1,1}} \| \mathbf{J}(\nabla u^{m}) - \mathbf{J}(\nabla u^{m-1}) \|_{L^{p}}\ ds.
\]
Using the fact that the sequence is uniformly bounded, and the fact that $\mathbf{J}$ is differentiable, we obtain
\[
\| \mathbf{J}(\nabla u^{m}) - \mathbf{J}(\nabla u^{m-1}) \|_{L^{p}} \leq C \| \nabla u^{m} - \nabla u^{m-1} \|_{L^{p}}
\]
for some constant $C$ by the Mean Value Theorem.  Using the inductive hypothesis and equation \eqref{kest2}, we obtain
\begin{align*}
\| u^{m+1} - u^{m} \|_{W^{1,p}} & \leq C \delta^{\frac{m}{2}}\int_{0}^{t} (t-s)^{-\frac{1}{2}} + (t-s)^{-\frac{1}{4}}\ ds \\
& \leq C \delta^{\frac{m}{2}} \left( \delta^{\frac{1}{2}} + \delta^{\frac{3}{4}} \right) \\
& \leq C \delta^{\frac{m+1}{2}},
\end{align*}
as desired.  As before, we observe that the above computations also hold for $p = \infty$.  This completes the proof of the proposition.
\end{proof}
\begin{cor}[Existence]\label{existence}
The sequence $\{ u_{n} \}_{n = 0}^{\infty}$ converges to a solution $u$ of equation \eqref{integraleq}.
\end{cor}
The next step is to show uniqueness.
\begin{prop}[Uniqueness]\label{uniqua}
The strong $W^{1,p} \cap W^{1,\infty}$ solution to equation \eqref{MBE} is unique.
\end{prop}
\begin{proof}
Suppose $u$ and $v$ are solutions to equation \eqref{integraleq} with initial data $u_{0}$.  By proceeding as we did above, we may use equation \eqref{integraleq} and the properties of $\mathbf{J}$ and $k_{t}$ to obtain the estimate
\begin{align*}
& \| u(t) - v(t) \|_{W^{1,p} \cap W^{1,\infty}} \\
& \qquad \qquad \quad \leq \int_{0}^{t} \left[(t-s)^{-\frac{1}{2}} + (t-s)^{-\frac{1}{4}} \right] \| u(s) - v(s) \|_{W^{1,p} \cap W^{1,\infty}}\ ds.
\end{align*}
Applying Proposition \ref{Bihari}, we obtain that
\[
\| u - v \|_{L^{\infty}_{t} (W^{1,p} \cap W^{1,\infty})} \leq 0.
\]
Uniqueness follows.
\end{proof}

To complete the proof of Proposition \ref{local}, we must also show that the solution map is continuous.
\begin{prop}\label{continuous}
The solution map $u_{0} \mapsto u$ is continuous.
\end{prop}
\begin{proof}
It suffices to show that there exists a constant $C > 0$ such that
\[
\| u - v \|_{L^{\infty}_{t}(W^{1,p} \cap W^{1,\infty})_{x}} \leq C \| u_{0} - v_{0} \|_{W^{1,p}}
\]
for solutions $u, v$ of equation \eqref{integraleq} corresponding to initial data $u_{0}, v_{0}$ respectively.  Since $u$ and $v$ satisfy equation \eqref{integraleq}, we have
\begin{align*}
\| u - v \|_{W^{1,p}} & \leq \| k_{t} * (u_{0} - v_{0}) \|_{W^{1,p}} \\
& \quad + \left\| \int_{0}^{t} \nabla k_{t - s} * ( \mathbf{J}(\nabla u) - \mathbf{J}(\nabla v))\ ds \right\|_{W^{1,p}}.
\end{align*}
Proceeding as above, we have
\[
\| k_{t} * (u_{0} - v_{0}) \|_{W^{1,p}} \leq \| u_{0} - v_{0} \|_{W^{1,p}}
\]
and
\[
\left\| \int_{0}^{t} \nabla k_{t - s} * ( \mathbf{J}(\nabla u) - \mathbf{J}(\nabla v))\ ds \right\|_{W^{1,p}} \leq C(t^{3/4} + t^{1/2})\| u - v \|_{L^{\infty}_{t}W^{1,p}_{x}}.
\]
We may thus conclude that
\[
\| u - v \|_{L^{\infty}_{t}W^{1,p}_{x}} \leq \| u_{0} - v_{0} \|_{W^{1,p}} + C(\delta^{3/4} + \delta^{1/2})\| u - v \|_{L^{\infty}_{t} W^{1,p}_{x}},
\]
which we rewrite as
\begin{equation}\label{welldef}
\| u - v \|_{L^{\infty}_{t}W^{1,p}_{x}} \leq \frac{1}{1 - C(\delta^{3/4} + \delta^{1/2})} \| u_{0} - v_{0} \|_{W^{1,p}}.
\end{equation}
Of course, this requires that
\[
C(\delta^{3/4} + \delta^{1/2}) < 1
\]
so that the right-hand side of equation \eqref{welldef} is well-defined.  In this case, we may conclude that the solution map $u_{0} \mapsto u$ is continuous.
\end{proof}
\noindent Proposition \ref{local} follows immediately from Corollary \ref{existence} and Propositions \ref{uniqua} and \ref{continuous}.

\section{Gradient Bounds on Solutions}\label{gradboundsec}

In this section, we will prove Theorem \ref{mainthm2}.  To begin, define
\[
M = \sup_{t \in [0,T]} \left\{ \| \nabla u(t) \|_{L^{p}_{x}} + t^{\frac{d}{4p}} \| \nabla u(t) \|_{L^{\infty}_{x}} \right\}.
\]
We will show that if the gradient of $u_{0}$ is sufficiently small, we will have $$M < \infty,$$ from which the desired result will follow.

Observe that
\[
\| \nabla u(t) \|_{L^{\infty}_{x}} = \| u(t) \|_{\dot{W}^{1,\infty}_{x}},
\]
so that
\begin{align*}
\| \nabla u \|_{L^{\infty}_{x}} & = \| k_{t} * u_{0} \|_{\dot{W}^{1,\infty}_{x}} + \int_{0}^{t} \| \nabla k_{t-s} * \mathbf{J}(\nabla u(s)) \|_{\dot{W}^{1,\infty}_{x}}\ ds \\
& \leq \| k_{t} \|_{L^{p'}} \| \nabla u_{0} \|_{L^{p}} + \int_{0}^{t} \| D^{2} k_{t-s} \|_{L^{1}} \| \nabla u(s) \|_{L^{\infty}_{x}}^{q}\ ds.
\end{align*}
Here, $p'$ is the conjugate exponent to $p$:
\[
\frac{1}{p} + \frac{1}{p'} = 1
\]
Thus, the above inequaility simplifies to
\begin{align*}
\| \nabla u \|_{L^{\infty}_{x}} & \leq C t^{-\frac{d}{4p}} \| \nabla u_{0} \|_{L^{p}} + C \int_{0}^{t} (t-s)^{-\frac{1}{2}} \| \nabla u(s) \|_{L^{\infty}_{x}}^{q}\ ds.
\end{align*} 
Multiplying by $t^{\frac{d}{4p}}$ yields
\begin{align*}
t^{\frac{d}{4p}}\| \nabla u \|_{L^{\infty}_{x}} & \leq C \| \nabla u_{0} \|_{L^{p}} + C t^{\frac{d}{4p}}\int_{0}^{t} (t-s)^{-\frac{1}{2}} \| \nabla u(s) \|_{L^{\infty}_{x}}^{q}\ ds \\
& \leq C \| \nabla u_{0} \|_{L^{p}} \\
& \quad + C t^{\frac{d}{4p}} \left(\int_{0}^{t} (t-s)^{-\frac{1}{2}} s^{-\frac{dq}{4p}} \ ds\right) \left(\sup_{s \in [0,T]} s^{\frac{d}{4p}}\| \nabla u(s) \|_{L^{\infty}_{x}}\right)^{q}.
\end{align*}
Applying Lemma \ref{beta}, this reduces to
\begin{align*}
t^{\frac{d}{4p}}\| \nabla u \|_{L^{\infty}_{x}} & \leq C \| \nabla u_{0} \|_{L^{p}} + C t^{\frac{1}{2}-\frac{d(q-1)}{4p}} \left(\sup_{s \in [0,T]} s^{\frac{d}{4p}}\| \nabla u(s) \|_{L^{\infty}_{x}}\right)^{q} \\
& \leq C \| \nabla u_{0} \|_{L^{p}} + C t^{\frac{1}{2} - \frac{d(q-1)}{4p}} M^{q}.
\end{align*}
Observe that the right-hand side of this inequality will be independent of $t$ if
\[
q = 1 + \frac{2p}{d}.
\]
This then becomes
\begin{equation}\label{gradest1}
t^{\frac{d}{4p}} \| \nabla u(t) \|_{L^{\infty}_{x}} \leq C \| \nabla u_{0} \|_{L^{p}} + C M^{1 + \frac{2p}{d}}.
\end{equation}

By a similar computation, it is easy to show that
\[
\| \nabla u(t) \|_{L^{p}_{x}} \leq \| \nabla u_{0} \|_{L^{p}} + C \int_{0}^{t} (t-s)^{-\frac{1}{2}} \| |\nabla u(s)|^{q} \|_{L^{p}_{x}}\ ds.
\]
From the definition of $q$, we have
\[
\| |\nabla u(s)|^{q} \|_{L^{p}_{x}} \leq \| \nabla u (s) \|_{L^{\infty}_{x}}^{\frac{2p}{d}} \| \nabla u(s) \|_{L^{p}_{x}}.
\]
Thus, we obtain the estimate
\[
\| \nabla u(t) \|_{L^{p}_{x}} \leq \| \nabla u_{0} \|_{L^{p}} + C \int_{0}^{t} (t-s)^{-\frac{1}{2}} \| \nabla u (s) \|_{L^{\infty}_{x}}^{\frac{2p}{d}} \| \nabla u(s) \|_{L^{p}_{x}}\ ds.
\]
Note that the integral in the right-hand side can be estimated by
\begin{align*}
& \int_{0}^{t} (t-s)^{-\frac{1}{2}} \| \nabla u (s) \|_{L^{\infty}_{x}}^{\frac{2p}{d}} \| \nabla u(s) \|_{L^{p}_{x}}\ ds \\
& \qquad \qquad  = \int_{0}^{t} (t-s)^{-\frac{1}{2}} s^{-\frac{1}{2}}\left( s^{\frac{d}{4p}}\| \nabla u (s) \|_{L^{\infty}_{x}}\right)^{\frac{2p}{d}} \| \nabla u(s) \|_{L^{p}_{x}}\ ds \\
& \qquad \qquad \leq \left(\int_{0}^{t} (t-s)^{-\frac{1}{2}} s^{-\frac{1}{2}}\ ds\right) M^{1+\frac{2p}{d}}.
\end{align*}
Applying Lemma \ref{beta}, we thus obtain the estimate
\begin{equation}\label{gradest2}
\| \nabla u(t) \|_{L^{p}} \leq \| \nabla u_{0} \|_{L^{p}} + C M^{1+\frac{2p}{d}}.
\end{equation}

Combining equations \eqref{gradest1} and \eqref{gradest2}, we have
\[
\| \nabla u(t) \|_{L^{p}} + t^{\frac{d}{4p}} \| \nabla u(t) \|_{L^{\infty}} \leq C_{1} \| \nabla u_{0} \|_{L^{p}} + C_{2} M^{1 + \frac{2p}{d}}.
\]
Taking the supremum over $t \in [0,T]$ gives us
\[
M \leq C_{1} \| \nabla u_{0} \|_{L^{p}} + C_{2} M^{1 + \frac{2p}{d}}.
\]
It then follows from Lemma \ref{strauss} that
\[
M < C
\]
for some constant $C > 0$ whenever $\| \nabla u_{0} \|_{L^{p}}$ is sufficiently small.  An immediate consequence of this is that
\begin{equation}\label{decay}
\| \nabla u(t) \|_{L^{\infty}_{x}} \leq \frac{C}{t^{\frac{d}{4p}}} \quad \text{and} \quad \| \nabla u(t) \|_{L^{p}_{x}} \leq C
\end{equation}
for $t > 0$.  This completes the proof of Theorem \ref{mainthm2}.

\begin{cor}\label{interpolate}
Assume that $u$ is a solution to equation \eqref{MBE} for $q > 1$, and let $p < p_{\theta} < \infty$ be such that
\[
\frac{1}{p_{\theta}} = \frac{1-\theta}{p}.
\]
Then
\[
\| \nabla u(t) \|_{L^{p_{\theta}}} \leq C t^{-\frac{d\theta}{4p}}.
\]
\end{cor}
\begin{proof}
This result can be obtained by interpolating the two estimates in equation \eqref{decay}, which hold under the smallness assumption on $\nabla u_{0}$.
\end{proof}

\section{Global Well-posedness}\label{globalwell}

In this section, we show that the local solutions constructed in  section \ref{localwell} exist globally.  Let $u$ be such a solution, and suppose that $T^{*}$ is the maximal time of existence.  If $T^{*} = +\infty$, there is nothing to prove, so suppose that $T^{*} < +\infty$.  By a standard argument, it must be the case that
\[
\sup_{t \in [0,T^{*}]} \| u(t) \|_{W^{1,p} \cap W^{1,\infty}} = +\infty.
\]
We will show that this is not the case.

By the local theory in the previous section, it is a simple matter to see that $u = 0$ is the unique strong $W^{1,p} \cap W^{1,\infty}$ solution to the Cauchy problem \eqref{Cauchy} with $u_{0} = 0$.  Thus, we may assume that $u_{0} \neq 0$.  In light of Theorem \ref{mainthm2}, we need only estimate the growth of $\| u(t) \|_{L^{p}}$ and $\| u(t) \|_{L^{\infty}}$.

By proceeding as we did to obtain equation \eqref{welldef}, it is easy to see that $u$ satisfies
\[
\| u(t) \|_{L^{p}} \leq \| u_{0} \|_{L^{p}} + C \int_{0}^{t} (t-s)^{-\frac{1}{4}} \| \nabla u(s)  \|_{L^{pq}}^{q}\ ds.
\]
To this, we apply Corollary \ref{interpolate} with $p_{\theta} = pq > p$ to obtain
\[
\| \nabla u(s) \|_{L^{pq}} \leq Ct^{-\frac{d}{4p}\left( 1 - \frac{1}{q} \right)}.
\]
Thus, we obtain that
\[
\| u(t) \|_{L^{p}} \leq \| u_{0} \|_{L^{p}} + C \int_{0}^{t} (t-s)^{-\frac{1}{4}} s^{-\frac{d}{4p}(q-1)}\ ds.
\]
Since we assume that
\[
q = 1 + \frac{2p}{d},
\]
then this estimate simplifies to
\[
\| u(t) \|_{L^{p}} \leq \| u_{0} \|_{L^{p}} + C \int_{0}^{t} (t-s)^{-\frac{1}{4}} s^{-\frac{1}{2}}\ ds.
\]
Applying Lemma \ref{beta}, we obtain
\begin{equation}\label{roughest1}
\| u(t) \|_{L^{p}} \leq \| u_{0} \|_{W^{1, p}} + C t^{\frac{1}{4}},
\end{equation}
which we may express as
\begin{equation}\label{lpest}
\| u(t) \|_{L^{p}} \leq C(1+t)^{\frac{1}{4}}.
\end{equation}

Now, we turn to estimating the $L^{\infty}$ norm of $u$.  Proceeding as we have before, we have the estimate
\begin{align*}
\| u(t) \|_{L^{\infty}_{x}} & \leq \| u_{0} \|_{L^{\infty}} + C \int_{0}^{t} \| \nabla k_{t-s} \|_{L^{p'}} \| |\nabla u|^{q} \|_{L^{p}} \\
& \leq \| u_{0} \|_{W^{1,\infty}} + C \int_{0}^{t} (t-s)^{-\frac{d}{4p}-\frac{1}{4}} \| \nabla u(s) \|_{L^{pq}}^{q}\ ds \\
& \leq  \| u_{0} \|_{W^{1,\infty}} + C \int_{0}^{t} (t-s)^{-\frac{d}{4p}-\frac{1}{4}} s^{-\frac{1}{2}}\ ds \\
& \leq  \| u_{0} \|_{W^{1,\infty}} + C t^{\frac{3}{4}-\frac{d}{4p}}.
\end{align*}
Of course, this requires that
\[
p > \frac{d}{3},
\]
which is satisfied when
\[
q > \frac{5}{3}.
\]
This completes the proof of Theorem \ref{mainthm1}.

\section{Coarseness}\label{rough}

In this section, we prove Theorem \ref{mainthm3}.  Begin by applying Lemma \ref{gagl} to obtain the estimate
\[
\| u(t) \|_{L^{2}} \leq C \| \nabla u(t) \|_{L^{p}}^{1-\theta} \| \nabla u(t) \|_{L^{p}}^{\theta}
\]
for
\begin{equation}\label{choosetheta}
\theta = d\left( \frac{1}{p} - \frac{1}{2} \right).
\end{equation}
For this to be valid, we require
\[
0 < d\left( \frac{1}{p} - \frac{1}{2} \right) < 1
\]
and
\[
1 < p < 2,
\]
with $p$ as in Theorem \ref{mainthm1}.  This is satisfied when
\[
\max \left\{\frac{6+d}{2+d}, 1 + \frac{2}{d} \right\} < q < 1 + \frac{4}{d}.
\]
In this case, we apply inequality \eqref{lpest} and Theorem \ref{mainthm2} to see that
\[
\| u(t) \|_{L^{2}} \leq C (1+t)^{\frac{1-\theta}{4}}.
\]
Substituting $\theta$ from equation \eqref{choosetheta}, we obtain
\[
\| u(t) \|_{L^{2}} \leq C (1+t)^{\frac{1}{4} - \frac{d}{4p} + \frac{d}{8}},
\]
as desired.  This completes the proof of Theorem \ref{mainthm3}.

\end{document}